\documentclass{amsart}
\usepackage{amssymb, amsmath, amsthm}

\usepackage{color}

\title{Solution of Mumford's second problem}

\author{Julia Bernatska, Yaacov Kopeliovich}
\date{\today}
\address{University of Connecticut, Department of Mathematics}
\email{julia.bernatska@uconn.edu, yaacov.kopeliovich@uconn.edu}
\newtheorem{thm}{Theorem}

\newtheorem{lem}{Lemma}

\newtheorem{prop}{Proposition}
\theoremstyle{definition}
\newtheorem{defn}{Definition}

\newtheorem*{Probl}{Problem} 
\newtheorem{example}{Example}
\newtheorem{rmk}{Remark}

\newcommand{\res}{\operatorname{res}}
\renewcommand{\mod}{\operatorname{mod}}

\newcommand{\lcm}{\operatorname{lcm}}
\newcommand{\rme}{\textrm{e}}
\newcommand{\m}{\phantom{-}}

\allowdisplaybreaks[4]

\begin{document}
\begin{abstract}
A complete solution of Mumford's second problem about representation of theta derivatives with
rational characteristics in terms of theta constants with rational characteristics is found.
An explicit formula for computing such an expression for theta derivative with an arbitrary rational characteristic
 is derived, and illustrated with examples.
Expressions for theta derivatives appear to be homogeneous of degree $3$ with respect to theta constants.
\end{abstract}

\maketitle

\section{Introduction}
One of the most famous results in the theory of theta functions is Jacobi's theta derivative identity: 
\begin{equation}\label{JacobiIdentity}
\theta'\big[{}^{1/2}_{1/2}\big](0,\tau) = -\pi
\theta\big[{}^0_0\big](0,\tau)
\theta\big[{}^{1/2}_{\ 0}\big](0,\tau)
\theta\big[{}^{\ 0}_{1/2}\big](0,\tau),
\end{equation}
where the notation $\theta'\big[{}^{\varepsilon'}_{\varepsilon}\big](0,\tau)
 = \frac{\partial}{\partial z}\theta\big[{}^{\varepsilon'}_{\varepsilon}\big](z,\tau) \big|_{z=0}$  is used.
This identity expresses the theta derivative with odd characteristic  through 
a product of theta constants with even characteristics. 
Apart from the elegance of this formula this result has applications in number theory,
for example implies that $\theta'\big[{}^{1/2}_{1/2}\big](0,\tau)$ 
is a cusp form of weight~$3/2$ with respect to the full modular group 
$\mathrm{SL}(2,\mathbb{Z}).$
Following  Jacobi's theta derivative identity Mumford posed the problem 
about expressing a theta derivative with rational characteristic
as a cubic polynomial of theta constants with rational characteristics, see \cite[p.117 Question (II)]{Mu}.
Here we formulate a wider 
\begin{Probl}
Express theta derivatives $\theta' \big[{}^{\varepsilon'}_{\varepsilon}\big](0,\tau)$ 
with rational characteristics through theta constants $\theta \big[{}^{\alpha'}_{\alpha}\big](0,\tau)$ 
with rational characteristics.
\end{Probl} 
An expression of the type we are interested in was obtained in \cite[Theorem 5]{FK} 
\begin{multline*}
6 \theta' \big[{}^{1/2}_{1/6} \big](0,\tau) 
\theta\big[{}^{1/6}_{1/6} \big](0,\tau) \theta\big[{}^{1/6}_{1/2} \big](0,\tau) \theta\big[{}^{1/6}_{5/6} \big](0,\tau) \\
= \theta' \big[{}^{1/2}_{1/2} \big](0,\tau) \Big( \theta^3 \big[{}^{1/6}_{1/2} \big](0,\tau) 
+ \rme^{-\imath \pi/3} \theta^3 \big[{}^{1/6}_{5/6} \big](0,\tau)
+ \rme^{\imath \pi/3} \theta^3 \big[{}^{1/6}_{1/6} \big](0,\tau) \Big).
\end{multline*}
Then  \cite{M1}, \cite{M2}, \cite{M3}, and \cite{Z} made progress on this problem. 
In \cite{M1} and then \cite{Z} the Jacobi triple product identity is used, and the derived relations
combine theta functions of the first and higher levels\footnote{We call $\theta(z, n\tau)$ a theta function of level $n$}.
In \cite{M1} only theta derivatives with characteristics $\big[{}^{1/2}_{1/4} \big]$,
$\big[{}^{1/2}_{1/8} \big]$, $\big[{}^{1/2}_{3/8} \big]$ are considered. In \cite{Z} representations
for all theta derivatives whose characteristics have denominators equal to $4$, $6$, an $8$ are presented,
moreover theta derivatives and theta constants are expressed in terms of Dedekind eta function.
Another technique, used in \cite{M2, M3}, was firstly proposed in papers \cite{FKop1,FKop2}.
It is based on the idea of constructing a special elliptic function, application of residue theorem to which leads to 
a theta constant identity. In \cite{M2,M3} by means of this technique 
some identities for theta derivatives were obtained. 
Different elliptic functions were employed to produce identities in which rational characteristics
have denominators $3$, $4$, $6$, $8$, or $10$. 

In the present paper a specially designed elliptic function leads to an identity which is 
applicable to theta derivatives with arbitrary rational or real characteristics. We call this identity fundamental,
and show how to solve it and obtain an expression for a theta derivative with an arbitrary rational characteristic. 
It is worth to note that the identities and expressions obtained in this paper are the simplest among the all known for now.

Therefore, the Problem is solved completely. Taking into consideration that expressions for theta derivatives are homogeneous
of degree $3$ with respect to theta constants, we claim that Mumford's second problem is solved as well.
As far as the authors know this level of generality for elliptic theta derivative identities
has not appeared before.

The paper is structured as follows. In Section 2  the definition of theta function and its fundamental properties
are recalled. Section 3 is devoted to the proof of the fundamental theta derivative identity, which
 connects two theta derivatives with rational characteristics $\big[{}^{\varepsilon'}_{\varepsilon}\big]$
and $\big[{}^{3\varepsilon'}_{3\varepsilon}\big]$. 
In Section 4 the identity is solved for theta derivatives, so an expression for theta derivative
with an arbitrary rational characteristic can be found. 
Numerous examples illustrate the proposed formulas.

\section{Preliminaries}
Recall definition of theta function and its fundamental properties.
\begin{defn}
For $z\in \mathbb{C}$ and $\tau$ from Siegel upper half-space we define 
\begin{equation} 
\theta(z,\tau)=\sum_{n\in \mathbb{Z}} \exp \big(\pi \imath n^2\tau + 2 \pi \imath n z\big)
\end{equation}
\end{defn}
\begin{defn}
For $\big[{}^{\varepsilon'}_\varepsilon \big]\in \mathbb{R}^2$ we define 
\begin{multline}
\theta\big[{}^{\varepsilon'}_{\varepsilon}\big](z,\tau) = \sum_{n\in \mathbb{Z}}
\exp\big(\pi\imath (n+\varepsilon')^2 \tau +  2 \pi \imath (n+\varepsilon')(z+\varepsilon) \big) \\
= \mathrm{e}^{\pi \imath {\varepsilon'}^2\tau + 2 \pi \imath \varepsilon' (z+\varepsilon)}
\theta(z + \tau \varepsilon'+\varepsilon,\tau) 
\end{multline}
\end{defn}
In what follows the transformation law under an element $n+m\tau$ of the period lattice 
is essentially used.
\begin{subequations}\label{ThetaArgRels}
\begin{prop}
With $m,n\in \mathbb{Z}$ the following hold:
\begin{equation}\label{Rel1}
\theta\big[{}^{\varepsilon'}_{\varepsilon}\big](z+\tau m + n,\tau)=
\mathrm{e}^{- \pi\imath m^2 \tau - 2\pi \imath ( m (z + \varepsilon) - n  \varepsilon' )}
\theta\big[{}^{\varepsilon'}_{\varepsilon}\big](z,\tau),
\end{equation}
\begin{equation}\label{Rel2}
\theta\big[{}^{-\varepsilon'}_{-\varepsilon}\big](-z,\tau)
=\theta\big[{}^{\varepsilon'}_{\varepsilon}\big](z,\tau),
\end{equation}
\begin{equation}\label{Rel3}
\theta\big[{}^{\varepsilon'+n'}_{\; \varepsilon+n}\big](z,\tau)
= \mathrm{e}^{2 \pi\imath n \varepsilon'} \theta\big[{}^{\varepsilon'}_{\varepsilon}\big](z,\tau),
\end{equation}
\begin{equation}\label{Rel4}
\theta\big[{}^{\varepsilon'}_{\varepsilon}\big](z+ \tau s' + s,\tau)
= \mathrm{e}^{-\pi\imath \tau {s'}^2 - 2\pi\imath s'(z+s+\varepsilon)}
\theta\big[{}^{\varepsilon'+s'}_{\; \varepsilon+s}\big](z,\tau).
\end{equation}
\end{prop} 
Combining \eqref{Rel2} and \eqref{Rel3} find
\begin{equation}\label{Rel5}
\theta\big[{}^{-\varepsilon'+n'}_{\; -\varepsilon+n}\big](z,\tau)
= \mathrm{e}^{- 2 \pi\imath n \varepsilon'} \theta\big[{}^{\varepsilon'}_{\varepsilon}\big](-z,\tau),
\end{equation}
\end{subequations}
which implies the following relations for theta constants
and theta derivatives:
\begin{subequations}\label{ThCRels}
\begin{equation}\label{TCRel}
\theta\big[{}^{\pm\varepsilon'+n'}_{\; \pm\varepsilon+n}\big](0,\tau)
= \mathrm{e}^{\pm 2 \pi\imath n \varepsilon'} \theta\big[{}^{\varepsilon'}_{\varepsilon}\big](0,\tau),
\end{equation}
\begin{equation}\label{TDRel}
\theta'\big[{}^{\pm\varepsilon'+n'}_{\; \pm\varepsilon+n}\big](0,\tau)
= \pm \mathrm{e}^{\pm 2 \pi\imath n \varepsilon'} \theta'\big[{}^{\varepsilon'}_{\varepsilon}\big](0,\tau),
\end{equation}
\end{subequations}

Taking derivative of \eqref{Rel4} with respect to $z$ obtain
\begin{subequations}\label{ThetaDerRels}
\begin{multline}
\theta'\big[{}^{\varepsilon'}_{\varepsilon}\big](z+ \tau s' + s,\tau)
= \mathrm{e}^{-\pi\imath \tau {s'}^2 - 2\pi\imath s'(z+s+\varepsilon)} \Big(
\theta'\big[{}^{\varepsilon'+s'}_{\; \varepsilon+s}\big](z,\tau) \\
- 2\pi\imath s' \theta\big[{}^{\varepsilon'+s'}_{\; \varepsilon+s}\big](z,\tau)\Big),
\end{multline}
\begin{multline}
\theta''\big[{}^{\varepsilon'}_{\varepsilon}\big](z+ \tau s' + s,\tau) 
= \mathrm{e}^{-\pi\imath \tau {s'}^2 - 2\pi\imath s'(z+s+\varepsilon)} \Big(
\theta''\big[{}^{\varepsilon'+s'}_{\; \varepsilon+s}\big](z,\tau) \\
- 4\pi\imath s' \theta'\big[{}^{\varepsilon'+s'}_{\; \varepsilon+s}\big](z,\tau) 
- 4\pi^2 s'{}^2 \theta\big[{}^{\varepsilon'+s'}_{\; \varepsilon+s}\big](z,\tau)\Big).
\end{multline}
\end{subequations}

\begin{prop}\label{P:Roots}
$\theta\big[{}^{\varepsilon'}_{\varepsilon}\big](z,\tau)$ has a unique zero 
at $\tau(\tfrac{1}{2} -\varepsilon') + \tfrac{1}{2} - \varepsilon$
in the fundamental parallelogram with sides $1$ and $\tau$ 
corresponding to an elliptic curve. 
 \end{prop}
\begin{proof}
Indeed, from \eqref{Rel2} and \eqref{Rel3} with $\varepsilon=\frac{1}{2}$, $\varepsilon' =\frac{1}{2}$, 
and $n=1$, $n'=1$ find
\begin{gather*}
\theta\big[{}^{-1/2}_{-1/2}\big](-z,\tau) = \theta\big[{}^{1/2}_{1/2}\big](z,\tau),\\
\theta\big[{}^{-1/2}_{-1/2}\big](-z,\tau) = \theta\big[{}^{1/2-1}_{1/2-1}\big](-z,\tau)
= \mathrm{e}^{-\pi \imath} \theta\big[{}^{1/2}_{1/2}\big](-z,\tau).
\end{gather*}
Since $\theta\big[{}^{1/2}_{1/2}\big](z,\tau)$ is an odd function, 
thus, $\theta\big[{}^{1/2}_{1/2}\big](0,\tau) = 0$, recall that 
$\theta\big[{}^{1/2}_{1/2}\big](z,\tau) = \vartheta_3(z,\tau)$.
From \eqref{Rel4} with 
and $s=\frac{1}{2} - \varepsilon $, $s'=\frac{1}{2} - \varepsilon'$ obtain
\begin{equation*}
\theta\big[{}^{1/2}_{1/2}\big](z,\tau) = 
\mathrm{e}^{\pi\imath \tau (\varepsilon'-1/2)^2 
- 2\pi\imath (\varepsilon'-1/2)(z+1/2)}
\theta\big[{}^{\varepsilon'}_{\varepsilon}\big]\big(z+ \tau(\tfrac{1}{2} -\varepsilon')
 + \tfrac{1}{2} - \varepsilon,\tau\big)
\end{equation*}
Thus, the function $\theta\big[{}^{\varepsilon'}_{\varepsilon}\big]\big(z+ \tau(\tfrac{1}{2} -\varepsilon')
 + \tfrac{1}{2} - \varepsilon,\tau\big)$ has zero at $\tau(\tfrac{1}{2} -\varepsilon')
 + \tfrac{1}{2} - \varepsilon$.
\end{proof}

\section{The fundamental theta derivative identity} 
In this section we formulate a fundamental identity connecting theta derivatives and theta constants  
with rational characteristics. 
\begin{thm}\label{T:ThetaDerRel}
Let $\big[{}^{\varepsilon'}_{\varepsilon}\big] \in \mathbb{R}^{2}$ be any characteristic. Then
\begin{multline}\label{MainExpr}
 \theta^2\big[{}^{\varepsilon'}_{\varepsilon}\big](0,\tau) 
 \Big(3\theta\big[{}^{3\varepsilon' }_{3\varepsilon}\big](0,\tau)
 \theta'\big[{}^{\varepsilon'}_{\varepsilon}\big](0,\tau)
-  \theta'\big[{}^{3\varepsilon' }_{3\varepsilon}\big](0,\tau) 
\theta\big[{}^{\varepsilon'}_{\varepsilon}\big](0,\tau)   \Big)\\
= \mathrm{e}^{6\pi \imath \varepsilon'} \theta'\big[{}^{1/2}_{1/2}\big](0,\tau)
\theta^3\big[{}^{\frac{1}{2} - 2\varepsilon'}_{\frac{1}{2} - 2\varepsilon}\big](0,\tau).
\end{multline}
\end{thm}
\begin{proof} 
Let $\big[{}^{\varepsilon'}_{\varepsilon}\big]$, $\big[{}^{\delta'}_{\delta}\big] \in \mathbb{R}^{2}$ be characteristics such that $(2\varepsilon+\delta) =1/2$ and $(2\varepsilon'+\delta') = 1/2 $. 
Then introduce the following elliptic function
\begin{equation} 
f(z)=\frac{\theta\big[{}^{1/2}_{1/2}\big](z,\tau)^3}
{\theta\big[{}^{\varepsilon'}_{\varepsilon}\big](z,\tau)^2 \theta\big[{}^{\delta'}_{\delta}\big](z,\tau)}.
\end{equation} 
Hence, by the residue theorem we have
\begin{equation*}
\sum \res(f)=0,
\end{equation*}
where the sum is taken over all poles in the fundamental domain.
 
As shown in  Proposition~\ref{P:Roots}
function $f$ has a simple pole at $ \tau(\tfrac{1}{2} -\delta') + \tfrac{1}{2} - \delta$,
 and a pole of order $2$ at $ \tau(\tfrac{1}{2} -\varepsilon')
 + \tfrac{1}{2} - \varepsilon$.
It is well-known that the residue at $a$ of a function of the form $h(z)/g(z)$ such that $h(a)\neq 0$,
$g(a)=0$ and $g'(a)\neq 0$ is computed by the formula 
\begin{subequations}\label{ResQuotient}
\begin{equation}\label{Res1Pole}
\res_{a}\bigg(\frac{h(z)}{g(z)}\bigg) = \frac{h(z_0)}{g'(z_0)}.
\end{equation}
Find a similar expression for the residue at $a$ of a function
of the form $h(z)/g(z)^2$
such that $h(a)\neq 0$, $g(a)=0$ and $g'(a)\neq 0$. 
First, note that $g(z) = (z-a)g_1(z)$ since $a$ is a simple zero of $g$. Then by the formula for a residue
we find
\begin{multline}\label{Res2Pole}
\res_{a}\bigg(\frac{h(z)}{g(z)^2}\bigg)=\frac{d}{dz}\bigg(\frac{h(z)}{g_1(z)^2}\bigg)\bigg|_{z=a} 
= \bigg(\frac{h'(z)}{g_1(z)^2}-\frac{2 g'_1(z) h(z)}{g_1(z)^3}\bigg)\bigg|_{z=a} \\
=\frac{h'(a)}{g'(a)^2} - h(a)\frac{g''(a)}{g'(a)^3}.
\end{multline} 
Here the fact that $g_1(a)=g'(a)$ and $g'_1(a)= \frac{1}{2} g''(a)$ was taken into account.
\end{subequations}

Using \eqref{ResQuotient} compute residues of $f$ in the fundamental domain.
By \eqref{Res2Pole} residue at $\tau(\frac{1}{2} -\varepsilon') + \frac{1}{2} - \varepsilon$
is the following:
\begin{multline}\label{Res2PoleExpr}
\res_{\tau(\frac{1}{2} -\varepsilon') + \frac{1}{2} - \varepsilon} (f) \\
= \frac{1}{\theta'\big[{}^{\varepsilon'}_{\varepsilon}\big](z,\tau)^2}
\dfrac{\partial}{\partial z}\bigg( \dfrac{\theta\big[{}^{1/2}_{1/2}\big](z,\tau)^3}
{\theta\big[{}^{\delta'}_{\delta}\big](z,\tau)}\bigg)
- \dfrac{\theta\big[{}^{1/2}_{1/2}\big](z,\tau)^3 \theta''\big[{}^{\varepsilon'}_{\varepsilon}\big](z,\tau)}
{\theta\big[{}^{\delta'}_{\delta}\big](z,\tau) \theta'\big[{}^{\varepsilon'}_{\varepsilon}\big](z,\tau)^3}
\Bigg|_{z=\tau(\frac{1}{2} -\varepsilon') + \frac{1}{2} - \varepsilon} \\
= \frac{\mathrm{e}^{- 2\pi \imath (\frac{1}{2} -\varepsilon')(\frac{3}{2} - 2\varepsilon - \delta) - 6 \pi \imath \varepsilon'} 
\theta\big[{}^{\varepsilon'}_{\varepsilon}\big](0,\tau)^2}{
\theta\big[{}^{\frac{1}{2} - \varepsilon' + \delta'}_{\; \frac{1}{2} - \varepsilon + \delta}\big](0,\tau)^2
\theta'\big[{}^{1/2}_{1/2}\big](0,\tau)^3} 
\Bigg( \bigg(
3\theta\big[{}^{\frac{1}{2} - \varepsilon' + \delta'}_{\; \frac{1}{2} - \varepsilon + \delta}\big](0,\tau)
\Big(-\theta'\big[{}^{\varepsilon'}_{\varepsilon}\big](0,\tau)
-  2\pi \imath (\tfrac{1}{2} -\varepsilon') \theta\big[{}^{\varepsilon'}_{\varepsilon}\big](0,\tau)\Big) \\
- \Big(\theta'\big[{}^{\frac{1}{2} - \varepsilon' + \delta'}_{\; \frac{1}{2} - \varepsilon + \delta}\big](0,\tau) 
- 2\pi \imath (\tfrac{1}{2} -\varepsilon') 
\theta\big[{}^{\frac{1}{2} - \varepsilon' + \delta'}_{\; \frac{1}{2} - \varepsilon + \delta}\big](0,\tau) \Big)
\theta\big[{}^{\varepsilon'}_{\varepsilon}\big](0,\tau) \bigg) 
\theta'\big[{}^{1/2}_{1/2}\big](0,\tau) \\
+ 4 \pi \imath (\tfrac{1}{2} -\varepsilon') 
\theta\big[{}^{\frac{1}{2} - \varepsilon' + \delta'}_{\; \frac{1}{2} - \varepsilon + \delta}\big](0,\tau)
\theta\big[{}^{\varepsilon'}_{\varepsilon}\big](0,\tau) 
 \theta'\big[{}^{1/2}_{1/2}\big](0,\tau) \Bigg),
\end{multline} 
where \eqref{ThetaArgRels} and \eqref{ThetaDerRels} are applied.
Then compute residue at $\tau(\frac{1}{2} -\delta') + \frac{1}{2} - \delta$
\begin{multline}\label{Res1PoleExpr}
\res_{\tau(\frac{1}{2} -\delta') + \frac{1}{2} - \delta} (f)
= \frac{\theta\big[{}^{1/2}_{1/2}\big](z,\tau)^3}
{\theta\big[{}^{\varepsilon'}_{\varepsilon}\big](z,\tau)^2 \theta' \big[{}^{\delta'}_{\delta}\big](z,\tau)}
\bigg|_{z=\tau(\frac{1}{2} -\delta') + \frac{1}{2} - \delta} \\
= \mathrm{e}^{- 2\pi \imath (\frac{1}{2} -\delta')(\frac{3}{2} - 2 \varepsilon  -\delta) + 2\pi \imath(1 - 2\varepsilon' - \delta')} 
\frac{\theta\big[{}^{\delta'}_{\delta}\big](0,\tau)^3}
{\theta\big[{}^{\frac{1}{2} - \varepsilon' + \delta'}_{\; \frac{1}{2} - \varepsilon +\delta}\big](0,\tau)^2
\theta'\big[{}^{1/2}_{1/2}\big](0,\tau)}.
\end{multline}

The residue theorem gives
\begin{equation}
\res_{\tau(\frac{1}{2} -\varepsilon') + \frac{1}{2} - \varepsilon} (f) 
+ \res_{\tau(\frac{1}{2} -\delta') + \frac{1}{2} - \delta} (f) = 0
\end{equation}
or with \eqref{Res1PoleExpr} and \eqref{Res2PoleExpr}
taken into account
\begin{multline}\label{edExpr}
-\theta^2\big[{}^{\varepsilon'}_{\varepsilon}\big](0,\tau)
\Big(3\theta\big[{}^{\frac{1}{2} - \varepsilon' + \delta'}_{\; \frac{1}{2} - \varepsilon + \delta}\big](0,\tau)
 \theta'\big[{}^{\varepsilon'}_{\varepsilon}\big](0,\tau) 
+ \theta'\big[{}^{\frac{1}{2} - \varepsilon' + \delta'}_{\; \frac{1}{2} - \varepsilon + \delta}\big](0,\tau) 
\theta\big[{}^{\varepsilon'}_{\varepsilon}\big](0,\tau)  \Big) \\
+ \mathrm{e}^{- 2\pi \imath (\varepsilon' -\delta')(\frac{3}{2} - 2 \varepsilon  -\delta)
+ 2\pi \imath(1 +\varepsilon' - \delta') } 
\theta^3\big[{}^{\delta'}_{\delta}\big](0,\tau) \theta'\big[{}^{1/2}_{1/2}\big](0,\tau) = 0.
\end{multline}
Since $2\varepsilon + \delta = 1/2$ and $2\varepsilon' + \delta' = 1/2$, 
and $\theta\big[{}^{1 - 3\varepsilon' }_{1 - 3\varepsilon}\big](0,\tau)
= \mathrm{e}^{-6\pi \imath \varepsilon'} \theta\big[{}^{3\varepsilon' }_{3\varepsilon}\big](0,\tau)$ 
as follows from relations \eqref{Rel3},
we obtain the fundamental theta identity \eqref{MainExpr}.
\end{proof}

\begin{rmk}
Note that in Theorem~\ref{T:ThetaDerRel} 
characteristic $\big[{}^{\varepsilon'}_{\varepsilon}\big]$ contains real numbers,
and the fundamental identity \eqref{MainExpr} holds for an arbitrary real characteristic.
However, we are interested in rational characteristics seeing how to solve such identities for
theta derivatives, which is explained in the next section. 
\end{rmk}

\section {Theta derivatives with rational characteristics} 
The fundamental theta derivative identity \eqref{MainExpr}
allows to solve completely the Problem declared in Introduction. 
Below an explicit formula for calculating a theta derivative with an arbitrary rational characteristic is obtained,
and illustrated with  examples.

\subsection{Multiplication by $3$}
Let $\varepsilon$ be a rational number of the form $m/p$. Introduce 
the operator  $\mathfrak{T}$ of multiplication by $3$
such that $\mathfrak{T} m/p = (3 m \mod p)/p$,
and the $k$-th power is $\mathfrak{T}^k m/p = \big((3^k \cdot m) \mod p\big)/p$. 
Under the action of $\mathfrak{T}$ a set of proper rational numbers 
$\mathcal{P}(p) = \{m/p \mid 1 \leqslant m \leqslant p-1\}$
splits into orbits:
$$\mathcal{O}_\varepsilon =\{\mathfrak{T}^k \varepsilon \mid  0 \leqslant k < p-1\}.$$  

If $p$ is prime and non-divisible by $3$, by  Fermat's little theorem
$\mathfrak{T}^{p-1} m/p = m/p$. Thus, the action of $\mathfrak{T}$ is periodic  
with period equal to $p-1$ or a divisor of $p-1$. For example, when $p=2$, $5$, $7$, or $17$
there exists only one non-trivial orbit $\mathcal{O}_{1/p}$ of cardinality $|\mathcal{O}_{1/p}|=p-1$, namely:
$1$, $4$, $6$, or $16$. However, at $p=11$ the set of proper rational numbers $\mathcal{P}(11)$
splits into two orbits: $\mathcal{O}_{1/11}$ and $\mathcal{O}_{2/11}$ of cardinality~$5$.
At $p=13$ set $\mathcal{P}(13)$
splits into four orbits: $\mathcal{O}_{1/13}$, $\mathcal{O}_{2/13}$, 
$\mathcal{O}_{4/13}$, and $\mathcal{O}_{7/13}$ of cardinality~$3$.

Let $p$ be non-prime, non-divisible by $3$. By Euler's theorem
with relatively prime $p$ and $3$  one has
$3^{\varphi(p)}\equiv 1 \mod p$,
where $\varphi$ is Euler's totient function computed as follows.
Let $p= p_1^{k_1} p_2^{k_2} \cdots p_r^{k_r}$, then
\begin{gather*}
\varphi(p) = p_1^{k_1-1} (p_1-1) p_2^{k_2-1} (p_2-1) \cdots p_r^{k_r-1}(p_r-1).
\end{gather*}
Thus, the action of $\mathfrak{T}$ is periodic  
with period equal to $\varphi(p)$ or a divisor of $\varphi(p)$, at that the sum of cardinalities of all orbits
is $p-1$. For example, if $p=4$, then $\varphi(4)=2$, and $\mathcal{P}(4)$ splits into two orbits:
$\mathcal{O}_{1/4}$ of cardinality $2$ and $\mathcal{O}_{2/4}$ of cardinality $1$. If $p=8$, then 
$\varphi(8)=4$, and $\mathcal{P}(8)$ splits into four orbits:
$\mathcal{O}_{1/8}$, $\mathcal{O}_{2/8}$, $\mathcal{O}_{5/8}$ of cardinality $2$ 
and $\mathcal{O}_{4/8}$ of cardinality $1$.

\begin{lem}\label{L:PeriodT}
The operator of multiplication by $3$ acts periodically on a rational characteristic $[{}^{\; n/q}_{\,m/p}]$
if $p$ and $q$ are not divisible by $3$.
\end{lem}
\begin{proof}
Let $\mathcal{O}_{m/p}$ and $\mathcal{O}_{n/q}$ be orbits of action of the introduced above 
operator $\mathfrak{T}$ of multiplication by $3$. As shown above, the  action of $\mathfrak{T}$
on the both rational numbers is periodic with periods $|\mathcal{O}_{m/p}|$ and $|\mathcal{O}_{n/q}|$, respectively.
Then the action of $\mathfrak{T}$ on characteristic $[{}^{\; n/q}_{\,m/p}]$ has period
$\mathfrak{t} = \lcm(|\mathcal{O}_{m/p}|,|\mathcal{O}_{n/q}|)$.
\end{proof}
In what follows with a rational characteristic $[{}^{\; n/q}_{\,m/p}]$ we associate an ordered set 
$$\mathcal{S}\big(\big[{}^{\; n/q}_{\,m/p} \big] \big) 
= \{\mathfrak{T}^k \big[{}^{\; n/q}_{\,m/p} \big] \mid 0 \leqslant k \leqslant \mathfrak{t} -1\}$$
of characteristics generated by the action of multiplication by $3$, 
and denote by $\mathfrak{t}$ the period of this action on such a set.

Now consider the case of $p$ divisible by $3$.
Start with the simplest cases. When $p=3$ there are two proper rational numbers 
$\{1/3,\, 2/3 \} = \mathcal{P}(3)$. Dividing $p$ by~$3$ come to $\tilde{p} = p/3 =1$,
at that $\mathcal{P}(1)=\{0\}$. Evidently, $\mathfrak{T}^k 0 = 0$ for all $k\geqslant 0$,
that is $0$  serves as a stationary value.
Acting on elements of $\mathcal{P}(3)$ by $\mathfrak{T}$ one gets $\{1/3,\,0,\,0,\, \dots\}$
and $\{2/3,\,0,\,0,\, \dots\}$, where each sequence stabilizes at $0$. Then assign
$\mathcal{O}_{1/3} = \{1/3,\,0\}$ and $\mathcal{O}_{2/3} = \{2/3,\,0\}$.
Next, consider  $p=6$, and $\tilde{p} = p/3 =2$. Recall
that $\mathcal{P}(2) = \{1/2\}$, and $\mathfrak{T}^k 1/2 = 1/2$ for all $k\geqslant 0$.
Thus, $1/2$  also serves as a stationary value, so $\mathcal{O}_{3/6} = \{1/2\}$. Another subset
contained into $\mathcal{P}(6)$ is $\mathcal{P}(3)$ whose elements behave under 
the action of $\mathfrak{T}$ as explained above, thus $\mathcal{O}_{2/6} = \{1/3,\,0\}$ 
and $\mathcal{O}_{4/6} = \{2/3,\,0\}$. Finally, by direct computation find
$\mathcal{O}_{1/6} = \{1/6,\,1/2\}$ and $\mathcal{O}_{5/6} = \{5/6,\,1/2\}$, where $1/2$
serves a stationary value. Summarizing, there exist two stationary values: $0$ and $1/2$,
the former finalizes orbits $\mathcal{O}_{m/3^a}$ with integer $a\geqslant 1$,
and the latter orbits $\mathcal{O}_{m/ (2\cdot 3^a)}$ with integer $a\geqslant 0$. 
Now
let $\tilde{p}$  be non-divisible by $3$, obtained from $p$ after
dividing by the maximal power of $3$. When $\tilde{p} > 3$ 
under the action of $\mathfrak{T}$ set $\mathcal{P}(\tilde{p})$ splits into distinct orbits,
and the action is periodic on each of them as proven above. Since $\mathcal{P}(\tilde{p}) \subset \mathcal{P}(p)$
numbers $m/p \in \mathcal{P}(p)$ with $m$ divisible by $3$ are distributed between these periodic orbits. 
Numbers $m/p$ with $m$ divisible by the divisors of $p$ different from $3$
under repeated multiplication by~$3$ produce orbits which finalize at one 
of the mentioned stationary values.  
The action of multiplication by~$3$ on other numbers $m/p\in \mathcal{P}(p)\backslash \mathcal{P}(\tilde{p})$ 
such that $m$ and $p$ are relatively prime 
finalizes at one of the periodic orbit of $\mathcal{P}(\tilde{p})$. 
For example, if $p=15$ we have a periodic orbit of $\mathcal{P}(5)$, which is
$\mathcal{O}_{3/15} = \{1/5,\,3/5,\,4/5,\,2/5\}$, 
orbits $\mathcal{O}_{5/15} = \{1/3,\,0\}$, $\mathcal{O}_{10/15} = \{2/3,\,0\}$, 
finalizing at a stationary value, and eight orbits:  $\mathcal{O}_{1/15}$,
$\mathcal{O}_{2/15}$, $\mathcal{O}_{4/15}$, $\mathcal{O}_{7/15}$, $\mathcal{O}_{8/15}$,
$\mathcal{O}_{11/15}$, $\mathcal{O}_{13/15}$, $\mathcal{O}_{14/15}$
 finalizing at the periodic orbit $\mathcal{O}_{3/15}$,
 namely: $\mathcal{O}_{1/15}=\{1/15,\,1/5,\,3/5,\,4/5,\,2/5\}$,
$\mathcal{O}_{2/15}=\{2/15,\,2/5,\,1/5,\,3/5,\,4/5\}$ and so on.

\subsection{Expressions for theta derivatives}
Assume $\varepsilon = m/p$ and take into account the relation
\begin{align*}
&\theta \big[{}^{3^{k} \varepsilon'}_{3^{k} m/p}\big] = 
\exp \big(2 \pi \imath \lfloor 3^{k} m/p \rfloor \mathfrak{T}^{k} \varepsilon' \big)
\theta \big[{}^{\mathfrak{T}^{k} \varepsilon'}_{\mathfrak{T}^{k} m/p}\big],
\end{align*}
and the same for derivatives due to \eqref{Rel3}.
By $\lfloor \cdot \rfloor$ the floor function 
or the integer part of a positive rational number is denoted. 
Here and in what follows we omit the arguments $(0,\tau)$ of theta constants and theta derivatives.
Now replace $[{}^{\varepsilon'}_{\varepsilon}]$ by
$[{}^{\mathfrak{T}^{k} \varepsilon'}_{\mathfrak{T}^{k} \varepsilon}]$ in 
\eqref{MainExpr}, and obtain  the following
\begin{multline}\label{MainExpr2}
3\theta\big[{}^{\mathfrak{T}^{k+1}\varepsilon' }_{\mathfrak{T}^{k+1}m/p}\big]
 \theta'\big[{}^{\mathfrak{T}^k \varepsilon'}_{\mathfrak{T}^k m/p}\big]
-  \theta'\big[{}^{\mathfrak{T}^{k+1} \varepsilon' }_{\mathfrak{T}^{k+1} m/p}\big]
\theta\big[{}^{\mathfrak{T}^k \varepsilon'}_{\mathfrak{T}^k m/p}\big]  \\
= \mathrm{e}^{6\pi \imath (1- \lfloor 3^{k} m/p \rfloor) \mathfrak{T}^{k} \varepsilon'
- 2 \pi \imath \lfloor 3^{k+1} m/p \rfloor \mathfrak{T}^{k+1} \varepsilon'} 
\theta'\big[{}^{1/2}_{1/2}\big]
\theta^3\big[{}^{\frac{1}{2} - 2 \cdot 3^k \varepsilon'}_{\frac{1}{2} - 2 \cdot 3^k m/p}\big]
\theta^{-2}\big[{}^{\mathfrak{T}^k \varepsilon'}_{\mathfrak{T}^k m/p}\big].
\end{multline}
It is convenient to compute $\theta'\big[{}^{1/2}_{1/2}\big]$ from the Jacobi identity \eqref{JacobiIdentity},
though we keep this theta derivative for brevity.

\begin{thm}
Theta derivatives with characteristics from
the ordered set
$\mathcal{S}\big( \big[{}^{\ \ \varepsilon'}_{\,m/p}\big] \big) 
= \{\mathfrak{T}^k  \big[{}^{\ \ \varepsilon'}_{\,m/p}\big] \mid 0 \leqslant k \leqslant \mathfrak{t} -1\}$
on which multiplication by $3$ acts periodically with period $\mathfrak{t}$ are expressed as follows
\begin{multline}\label{ThetaDerSol}
\theta'\big[{}^{\mathfrak{T}^{k} \varepsilon'}_{\mathfrak{T}^{k} m/p}\big]
= \bigg((3^{\mathfrak{t}}-1) \prod_{l=0}^{\mathfrak{t}-1} 
\theta\big[{}^{\mathfrak{T}^{l}\varepsilon' }_{\mathfrak{T}^{l}m/p}\big] \bigg)^{-1} 
\theta'\big[{}^{1/2}_{1/2}\big]
\theta\big[{}^{\mathfrak{T}^{k}\varepsilon' }_{\mathfrak{T}^{k}m/p}\big] \times \\  \times
\sum_{j=k}^{\mathfrak{t}-1+k}  3^{\mathfrak{t}-1+k-j} 
\mathrm{e}^{6\pi \imath (1- \lfloor 3^{j} m/p \rfloor) \mathfrak{T}^{j} \varepsilon'
- 2 \pi \imath \lfloor 3^{j+1} m/p \rfloor \mathfrak{T}^{j+1} \varepsilon'}  \times \\  \times
\theta^3\big[{}^{\frac{1}{2} - 2 \cdot 3^{j} \varepsilon'}_{\frac{1}{2} - 2\cdot 3^{j} m/p}\big]
\theta^{-2}\big[{}^{\mathfrak{T}^{j} \varepsilon'}_{\mathfrak{T}^{j} m/p}\big]
\prod_{\substack{l = k \\ l\neq j \\ l \neq j+1}}^{\mathfrak{t}-1+k} 
\theta\big[{}^{\mathfrak{T}^{l}\varepsilon' }_{\mathfrak{T}^{l}m/p}\big].
\end{multline}
\end{thm}
\begin{proof}
Each pair of characteristics from the set $\mathcal{S}\big( \big[{}^{\ \ \varepsilon'}_{\,m/p}\big]\big)$, namely:
$[{}^{\mathfrak{T}^k \varepsilon'}_{\mathfrak{T}^k m/p}]$
and $[{}^{\mathfrak{T}^{k+1} \varepsilon'}_{\mathfrak{T}^{k+1} m/p}]$ with $0 \leqslant k \leqslant \mathfrak{t}-1$,
gives rise to an equation \eqref{MainExpr2}. Therefore, we obtain a system of $\mathfrak{t}$ equations
which are linear with respect to unknowns
$\{x_{k+1} = \theta'\big[{}^{\mathfrak{T}^{k} \varepsilon'}_{\mathfrak{T}^{k} m/p}\big](0,\tau) \mid
0 \leqslant k \leqslant \mathfrak{t}-1\}$. The system has the form $A x = B$ with the vector $x=(x_i)$ of  unknowns,
the coefficient matrix $A$ whose non-zero entries are the following
\begin{align*}
&A_{i,i} = 3 a_{i},\ i=1,\,\dots,\, \mathfrak{t}-1, \qquad 
A_{\mathfrak{t},\mathfrak{t}} = 3a_0,\qquad
\text{where }
a_k = \theta\big[{}^{\mathfrak{T}^{k}\varepsilon' }_{\mathfrak{T}^{k}\varepsilon}\big],\\
&A_{i,i+1} = -a_{i-1},\ i=1,\,\dots,\, \mathfrak{t}-1, \qquad
A_{\mathfrak{t},1} = - a_{\mathfrak{t}-1},
\end{align*}
and other entries are zero;
and the constant vector $B$ with entries:
\begin{align*}
&B_{k+1} = b_{k} =
\mathrm{e}^{6\pi \imath (1- \lfloor 3^{k} m/p \rfloor) \mathfrak{T}^{k} \varepsilon'
- 2 \pi \imath \lfloor 3^{k+1} m/p \rfloor \mathfrak{T}^{k+1} \varepsilon'} \times \\ 
&\qquad\qquad\qquad \times \theta'\big[{}^{1/2}_{1/2}\big]
\theta^3\big[{}^{\frac{1}{2} - 2 \cdot 3^k \varepsilon'}_{\frac{1}{2} - 2 \cdot 3^k m/p}\big]
\theta^{-2}\big[{}^{\mathfrak{T}^k \varepsilon'}_{\mathfrak{T}^k m/p}\big],
\quad k=0,\,\dots,\, \mathfrak{t}-1.
\end{align*}
In what follows assume that $a_{\mathfrak{t}+k} \equiv a_{k}$ and $b_{\mathfrak{t}+k} \equiv b_{k}$ 
for any integer $k \geqslant 0$.
The system can be solved by Cramer's rule, since matrix $A$ is non-singular, at that 
\begin{align*}
&\Delta = \det A = (3^{\mathfrak{t}}-1) \prod_{k =0}^{\mathfrak{t}-1} a_{k},\\
&\Delta_{k+1} = a_{k} \sum_{j=k}^{\mathfrak{t} -1+k}  3^{\mathfrak{t}-1+k - j} b_j 
\prod_{\substack{l = k\\ l\neq j \\ l\neq j+1}}^{\mathfrak{t}-1+k} a_{l},
\quad k=0,\,\dots,\, \mathfrak{t}-1.
\end{align*}
The solution is unique and has the form \eqref{ThetaDerSol}.
\end{proof}

\begin{rmk}
If $\varepsilon = m/p >1/2$ or $\varepsilon' >1/2$ then 
with the help of \eqref{ThCRels}  characteristic~$\big[{}^{\varepsilon'}_{\varepsilon}\big]$ can be reduced to
the form $\big[{}^{\delta'}_{\delta}\big]$ where the both $\delta$ and $\delta'$ are less than $1/2$.
This will reduce the number of equations twice, and simplify expression \eqref{ThetaDerSol}. 
\end{rmk}
\begin{rmk}\label{Rmk3}
When the denominator of $\varepsilon$ or $\varepsilon'$ is divisible by $3$, 
the action of multiplication by $3$ on the set generated by this characteristic is not periodic, but comes to a stationary value.
In this case the system of equations \eqref{MainExpr2} is solved starting from the last equation, and theta derivatives with
a stationary value in characteristics should be found in advance, see Examples~\ref{E:p6} and ~\ref{E:p65}.
\end{rmk}

Now we illustrate how to apply formula \eqref{ThetaDerSol} with different examples.

\begin{example}
The simplest case $p=2$ leads to the identical zero on the both sides of the equality
due to functions $\theta\big[{}^{1/2}_{1/2}\big]$, $\theta'\big[{}^{\ 0}_{1/2}\big]$, 
$\theta'\big[{}^{1/2}_{\ 0}\big]$ are odd.
\end{example}

\begin{example}
In the case of $p=3$ there are two orbits: $\mathcal{O}_{1/3} = \{1/3,\,0\}$ 
and $\mathcal{O}_{2/3} = \{2/3,\, 0\}$, 
which produce four sets of characteristics:
\begin{gather*}
\big\{\big[{}^{\ 0}_{1/3}\big],\, \big[{}^{0}_{0}\big] \big\},\quad
\big\{\big[{}^{1/3}_{\ 0}\big],\, \big[{}^{0}_{0}\big] \big\},\quad
\big\{\big[{}^{1/3}_{1/3}\big],\, \big[{}^{0}_{0}\big]\big\},\quad
\big\{ \big[{}^{1/3}_{2/3}\big],\, \big[{}^{0}_{0}\big] \big\}.
\end{gather*}
Since $\theta'\big[{}^{0}_{0}\big]$ is odd, the
following expressions for theta derivatives are derived directly from \eqref{MainExpr}:
\begin{align*}
&\theta'\big[{}^{\ 0}_{1/3}\big] = - \frac{\pi}{3} \theta\big[{}^{1/2}_{\ 0}\big] \theta\big[{}^{\;0}_{1/2}\big]
\theta^3 \big[{}^{1/2}_{1/6}\big] \theta^{-2}\big[{}^{\ 0}_{1/3}\big],\\
&\theta'\big[{}^{1/3}_{\ 0}\big] = \m \frac{\pi}{3} \theta\big[{}^{1/2}_{\ 0}\big] \theta\big[{}^{\;0}_{1/2}\big]
\theta^3 \big[{}^{1/6}_{1/2}\big] \theta^{-2}\big[{}^{1/3}_{\ 0}\big],\\
&\theta'\big[{}^{1/3}_{1/3}\big] = - \frac{\pi}{3} \theta\big[{}^{1/2}_{\ 0}\big] \theta\big[{}^{\;0}_{1/2}\big]
\theta^3 \big[{}^{1/6}_{1/6}\big] \theta^{-2}\big[{}^{1/3}_{1/3}\big],\\
&\theta'\big[{}^{1/3}_{2/3}\big] = - \frac{\pi}{3} \theta\big[{}^{1/2}_{\ 0}\big] \theta\big[{}^{\;0}_{1/2}\big]
\theta^3 \big[{}^{1/6}_{5/6}\big] \theta^{-2}\big[{}^{1/3}_{2/3}\big].
\end{align*}
At the same time, due to \eqref{ThCRels} theta constants and derivatives with 
characteristics $\big[{}^{\ 0}_{2/3}\big]$, $\big[{}^{2/3}_{\ 0}\big]$ $\big[{}^{2/3}_{2/3}\big]$, 
$\big[{}^{2/3}_{1/3}\big]$ relate to ones with characteristics $\big[{}^{\ 0}_{1/3}\big]$, 
$\big[{}^{1/3}_{\ 0}\big]$ $\big[{}^{1/3}_{1/3}\big]$, $\big[{}^{1/3}_{2/3}\big]$, respectively.
Therefore, theta identities with 
$\theta'\big[{}^{\ 0}_{2/3}\big]$, $\theta'\big[{}^{2/3}_{\ 0}\big]$,
$\theta'\big[{}^{2/3}_{2/3}\big]$,  $\theta'\big[{}^{2/3}_{1/3}\big]$ are found as well.

One can combine orbit $\mathcal{O}_{1/3}$ with $\mathcal{O}_{1/2}=\{1/2\}$,
and consider sets of characteristics
$\big\{\big[{}^{1/2}_{1/3}\big],\, \big[{}^{1/2}_{\ 0}\big] \big\}$, 
$\big\{\big[{}^{1/3}_{1/2}\big],\, \big[{}^{\ 0}_{1/2}\big] \big\}$.
Then \eqref{MainExpr} implies 
\begin{align*}
&\theta'\big[{}^{1/2}_{1/3}\big] = - \frac{\pi}{3} \theta\big[{}^{0}_{0}\big] \theta\big[{}^{\;0}_{1/2}\big]
\theta^3 \big[{}^{1/2}_{1/6}\big] \theta^{-2}\big[{}^{1/2}_{1/3}\big],\\
&\theta'\big[{}^{1/3}_{1/2}\big] = - \frac{\pi}{3} \theta\big[{}^{0}_{0}\big] \theta\big[{}^{1/2}_{\ 0}\big]
\theta^3 \big[{}^{1/6}_{1/2}\big] \theta^{-2}\big[{}^{1/3}_{1/2}\big].
\end{align*}
The expressions for theta derivatives in this example coincide 
with ones presented in \cite[subsect.\,5.1, 5.3, 5.5, 5.6]{M3}.
\end{example}

\begin{example}\label{Ex:p4}
In the case of $p=4$, two orbits exist: $\mathcal{O}_{1/4} = \{1/4,\, 3/4\}$,
$\mathcal{O}_{2/4} = \{1/2\}$.
Combining these orbits, write down several sets of characteristics:
\begin{gather*}
\big\{\big[{}^{\ 0}_{1/4}\big],\, \big[{}^{\ 0}_{3/4}\big] \big\},\quad
\big\{\big[{}^{1/4}_{\ 0}\big],\, \big[{}^{3/4}_{\ 0}\big] \big\},\quad
\big\{\big[{}^{1/4}_{1/4}\big],\, \big[{}^{3/4}_{3/4}\big]\big\},\quad
\big\{ \big[{}^{1/4}_{3/4}\big],\, \big[{}^{3/4}_{1/4}\big] \big\},\\
\big\{\big[{}^{1/2}_{1/4}\big],\, \big[{}^{1/2}_{3/4}\big] \big\},\quad
\big\{\big[{}^{1/4}_{1/2}\big],\, \big[{}^{3/4}_{1/2}\big]\big\}.
\end{gather*}
Each set  produces a systems of two equations of the form \eqref{MainExpr2}.
However, due to \eqref{ThCRels} each of these systems is reduced to one.
Therefore, it is more convenient to start computation from \eqref{MainExpr}.
The following expressions are obtained:
\begin{align*}
&\theta'\big[{}^{\ 0}_{1/4}\big]
 = - \frac{\pi}{4} \theta^4\big[{}^{1/2}_{\ 0}\big]
\theta\big[{}^0_0\big] \theta\big[{}^{\ 0}_{1/2 }\big]
\theta^{-3}\big[{}^{\ 0}_{1/4}\big],\\
&\theta'\big[{}^{1/4}_{\ 0}\big]
 = \m \imath \frac{\pi}{4} \theta^4\big[{}^{\ 0}_{1/2 }\big]
\theta\big[{}^0_0\big] \theta\big[{}^{1/2}_{\ 0}\big] 
\theta^{-3}\big[{}^{1/4}_{\ 0}\big], \\
&\theta'\big[{}^{1/4}_{1/4}\big]
 = -\frac{\pi}{4} \theta^4\big[{}^0_0\big]
 \theta\big[{}^{\ 0}_{1/2 }\big] \theta\big[{}^{1/2}_{\ 0}\big]
\theta^{-3}\big[{}^{1/4}_{1/4}\big], \\
&\theta'\big[{}^{1/4}_{3/4}\big]
 = \m \frac{ \pi}{4} \theta^4\big[{}^0_0\big]
 \theta\big[{}^{\ 0}_{1/2 }\big] \theta\big[{}^{1/2}_{\ 0}\big]
\theta^{-3}\big[{}^{1/4}_{3/4}\big],\\
&\theta'\big[{}^{1/4}_{1/2}\big] = - \imath \frac{ \pi}{4} 
 \theta^4 \big[{}^{\ 0}_{1/2 }\big] \theta\big[{}^0_0\big]
 \theta\big[{}^{1/2}_{\ 0}\big]
\theta^{-3}\big[{}^{1/4}_{1/2}\big], \\
& \theta'\big[{}^{1/2}_{1/4}\big] = - \frac{ \pi}{4} 
 \theta^4 \big[{}^{1/2}_{\ 0}\big] \theta\big[{}^0_0\big]
 \theta\big[{}^{\ 0}_{1/2}\big]
\theta^{-3}\big[{}^{1/2}_{1/4}\big].
\end{align*}
These expressions confirm the results of  \cite[subsect.\,3.1--3.2]{M3}.
\end{example}

\begin{example} 
In the case of $p=5$
only one orbit $\mathcal{O}_{1/5}=\{1/5,\, 3/5,\, 4/5,\, 2/5\}$
exists. Let a collection of characteristics be 
$\big\{\big[{}^{1/5}_{2/5}\big],\, \big[{}^{3/5}_{1/5}\big],\,\big[{}^{4/5}_{3/5}\big],\, 
\big[{}^{2/5}_{4/5}\big]\big\}$, 
then by formula \eqref{ThetaDerSol} obtain
\begin{multline*}
\theta'\big[{}^{1/5}_{2/5}\big] = \Big(80 \theta\big[{}^{1/5}_{2/5}\big] 
\theta\big[{}^{3/5}_{1/5}\big] \theta\big[{}^{4/5}_{3/5}\big] \theta\big[{}^{2/5}_{4/5}\big] \Big)^{-1}
\theta'\big[{}^{1/2}_{1/2}\big] \theta\big[{}^{1/5}_{2/5}\big] \times \\ \times
\Big(27 \theta^3 \big[{}^{\m 1/10}_{-3/10}\big]  \theta^{-2}\big[{}^{1/5}_{2/5}\big] 
\theta\big[{}^{4/5}_{3/5}\big] \theta\big[{}^{2/5}_{4/5}\big] 
+ 9 \rme^{-4 \pi \imath/5} \theta^3 \big[{}^{- 7/10}_{-19/10}\big]  \theta^{-2}\big[{}^{3/5}_{1/5}\big] 
\theta\big[{}^{1/5}_{2/5}\big] \theta\big[{}^{2/5}_{4/5}\big] \\
+ 3 \rme^{2 \pi \imath/5} \theta^3 \big[{}^{-31/10}_{-67/10}\big]  \theta^{-2}\big[{}^{4/5}_{3/5}\big] 
\theta\big[{}^{1/5}_{2/5}\big] \theta\big[{}^{3/5}_{1/5}\big] 
+ \rme^{- 2 \pi \imath/5} \theta^3 \big[{}^{-103/10}_{-211/10}\big]  \theta^{-2}\big[{}^{2/5}_{4/5}\big] 
\theta\big[{}^{3/5}_{1/5}\big] \theta\big[{}^{4/5}_{3/5}\big] \Big).
\end{multline*}
Using \eqref{TCRel}  find
\begin{align*}
&\theta\big[{}^{2/5}_{4/5}\big] = \rme^{4\pi \imath/5}  \theta \big[{}^{3/5}_{1/5}\big],&  
&\theta\big[{}^{4/5}_{3/5}\big] = \rme^{-2\pi \imath/5} \theta\big[{}^{1/5}_{2/5}\big],&\\
&\theta\big[{}^{\m 1/10}_{-3/10}\big] = \rme^{-\pi \imath/5}  \theta \big[{}^{1/10}_{7/10}\big],&  
&\theta\big[{}^{- 7/10}_{-19/10}\big] = \rme^{-6 \pi \imath/5} \theta\big[{}^{3/10}_{1/10}\big],&\\
&\theta\big[{}^{-31/10}_{-67/10}\big] = \rme^{6 \pi \imath/5}  \theta \big[{}^{1/10}_{7/10}\big],&  
&\theta\big[{}^{-103/10}_{-211/10}\big] = \rme^{3 \pi \imath/5} \theta\big[{}^{3/10}_{1/10}\big].&
\end{align*}
Thus, 
\begin{multline*}
\theta'\big[{}^{1/5}_{2/5}\big] = \Big(80 \rme^{2\pi \imath/5}  \theta^2\big[{}^{1/5}_{2/5}\big] 
\theta^2\big[{}^{3/5}_{1/5}\big] \Big)^{-1}
\theta'\big[{}^{1/2}_{1/2}\big] \theta\big[{}^{1/5}_{2/5}\big] \times \\ \times
\Big(27 \rme^{-\pi \imath/5}\theta^3 \big[{}^{1/10}_{7/10}\big]  
\theta^{-1}\big[{}^{1/5}_{2/5}\big]  \theta\big[{}^{3/5}_{1/5}\big] 
+ 9 \rme^{2 \pi \imath/5} \theta^3 \big[{}^{3/10}_{1/10}\big]  \theta^{-1}\big[{}^{3/5}_{1/5}\big] 
\theta\big[{}^{1/5}_{2/5}\big] \\
+ 3 \rme^{4 \pi \imath/5} \theta^3 \big[{}^{1/10}_{7/10}\big]  \theta^{-1}\big[{}^{1/5}_{2/5}\big] 
 \theta\big[{}^{3/5}_{1/5}\big] 
+ \rme^{-3 \pi \imath/5} \theta^3 \big[{}^{3/10}_{1/10}\big]  \theta^{-1}\big[{}^{3/5}_{1/5}\big] 
\theta\big[{}^{1/5}_{2/5}\big] \Big) \\
= \frac{1}{10} \theta'\big[{}^{1/2}_{1/2}\big]  
\Big(3 \rme^{- 3\pi \imath/5} \theta^3 \big[{}^{1/10}_{7/10}\big] \theta^{-2}\big[{}^{1/5}_{2/5}\big] 
\theta^{-1}\big[{}^{3/5}_{1/5}\big] 
+ \theta^3 \big[{}^{3/10}_{1/10}\big]  \theta^{-3}\big[{}^{3/5}_{1/5}\big]  \Big).
\end{multline*}
Formula \eqref{ThetaDerSol} for $\theta'\big[{}^{3/5}_{1/5}\big]$ reads as follows
\begin{multline*}
\theta'\big[{}^{3/5}_{1/5}\big] = \Big(80 \theta\big[{}^{1/5}_{2/5}\big] 
\theta\big[{}^{3/5}_{1/5}\big] \theta\big[{}^{4/5}_{3/5}\big] \theta\big[{}^{2/5}_{4/5}\big] \Big)^{-1}
\theta'\big[{}^{1/2}_{1/2}\big] \theta\big[{}^{3/5}_{1/5}\big] \times \\ \times
\Big(27 \rme^{-4 \pi \imath/5} \theta^3 \big[{}^{- 7/10}_{-19/10}\big]  \theta^{-2}\big[{}^{3/5}_{1/5}\big] 
\theta\big[{}^{1/5}_{2/5}\big] \theta\big[{}^{2/5}_{4/5}\big] 
+ 9 \rme^{2 \pi \imath/5} \theta^3 \big[{}^{-31/10}_{-67/10}\big]  \theta^{-2}\big[{}^{4/5}_{3/5}\big] 
\theta\big[{}^{1/5}_{2/5}\big] \theta\big[{}^{3/5}_{1/5}\big] \\
+ 3\rme^{- 2 \pi \imath/5} \theta^3 \big[{}^{-103/10}_{-211/10}\big]  \theta^{-2}\big[{}^{2/5}_{4/5}\big] 
\theta\big[{}^{3/5}_{1/5}\big] \theta\big[{}^{4/5}_{3/5}\big]
+\theta^3 \big[{}^{\m 1/10}_{-3/10}\big]  \theta^{-2}\big[{}^{1/5}_{2/5}\big] 
\theta\big[{}^{4/5}_{3/5}\big] \theta\big[{}^{2/5}_{4/5}\big]  \Big),
\end{multline*}
that gives
\begin{equation*}
\theta'\big[{}^{3/5}_{1/5}\big] 
= \frac{1}{10} \theta'\big[{}^{1/2}_{1/2}\big] 
\Big(\rme^{2\pi \imath/5}  \theta^3 \big[{}^{1/10}_{7/10}\big] \theta^{-3}\big[{}^{1/5}_{2/5}\big] 
+ 3 \theta^3 \big[{}^{3/10}_{1/10}\big] \theta^{-2}\big[{}^{3/5}_{1/5}\big] \theta^{-1}\big[{}^{1/5}_{2/5}\big]  \Big).
\end{equation*}
Applying \eqref{JacobiIdentity}, obtain
\begin{align*}
&\theta'\big[{}^{1/5}_{2/5}\big] = - \frac{\pi}{10} 
\theta\big[{}^{0}_{0}\big] \theta\big[{}^{\ 0}_{1/2}\big] \theta\big[{}^{1/2}_{\ 0}\big]  
\Big(3 \rme^{- 3\pi \imath/5} \theta^3 \big[{}^{1/10}_{7/10}\big]  
\theta^{-2}\big[{}^{1/5}_{2/5}\big] \theta^{-1}\big[{}^{3/5}_{1/5}\big] 
+ \theta^3 \big[{}^{3/10}_{1/10}\big]  
\theta^{-3}\big[{}^{3/5}_{1/5}\big]  \Big), \\
&\theta'\big[{}^{3/5}_{1/5}\big] = - \frac{\pi}{10} 
\theta\big[{}^{0}_{0}\big] \theta\big[{}^{\ 0}_{1/2}\big] \theta\big[{}^{1/2}_{\ 0}\big] 
\Big(\rme^{2\pi \imath/5}  \theta^3 \big[{}^{1/10}_{7/10}\big] \theta^{-3}\big[{}^{1/5}_{2/5}\big] 
+ 3 \theta^3 \big[{}^{3/10}_{1/10}\big] \theta^{-2}\big[{}^{3/5}_{1/5}\big]  \theta^{-1}\big[{}^{1/5}_{2/5}\big] \Big).
\end{align*}
\end{example}

\begin{example}\label{E:p6}
In the case of $p=6$ the new values, not considered in the previous examples,
are $1/6$, $5/6$ with orbits $\mathcal{O}_{1/6}=\{1/6,\,1/2\}$, $\mathcal{O}_{5/6}=\{5/6,\,1/2\}$. 
The following sets of characteristics are constructed from these orbits:
\begin{gather*}
\big\{\big[{}^{\ 0}_{1/6}\big],\, \big[{}^{\ 0}_{1/2}\big] \big\},\quad
\big\{\big[{}^{1/6}_{\ 0}\big],\, \big[{}^{1/2}_{\ 0}\big] \big\},\quad
\big\{\big[{}^{1/6}_{1/6}\big],\, \big[{}^{1/2}_{1/2}\big]\big\},\quad
\big\{ \big[{}^{1/6}_{5/6}\big],\, \big[{}^{1/2}_{1/2}\big] \big\}.
\end{gather*}
Here the fundamental identity  \eqref{MainExpr}  is again more convenient for computation,
and two expressions can be found:
\begin{align*}
&\theta'\big[{}^{\ 0}_{1/6}\big] = -\frac{\pi}{3} 
\theta\big[{}^{0}_{0}\big] \theta\big[{}^{1/2}_{\ 0}\big] 
\theta^3 \big[{}^{1/2}_{1/6}\big] \theta^{-2}\big[{}^{\ 0}_{1/6}\big],\\
&\theta'\big[{}^{1/6}_{\ 0}\big] = \m \frac{\pi}{3} 
\theta\big[{}^{0}_{0}\big] \theta\big[{}^{\ 0}_{1/2}\big] 
\theta^3 \big[{}^{1/6}_{1/2}\big] \theta^{-2}\big[{}^{1/6}_{\ 0}\big].
\end{align*}
Expressions for $\theta' \big[{}^{1/6}_{1/6}\big]$ and $\theta' \big[{}^{1/6}_{5/6}\big]$ 
can not be obtained from~\eqref{MainExpr} due to vanishing $\theta \big[{}^{1/2}_{1/2}\big]$
which serves as a coefficient at the desired theta derivatives.
In \cite[Subsect.\;6.1, 6.3]{M3} expressions for these theta derivatives are found from other elliptic functions.

One also can combine $\mathcal{O}_{1/6}$ with $\mathcal{O}_{1/2}$, $\mathcal{O}_{1/3}$, 
and  $\mathcal{O}_{1/4}$
\begin{gather*}
\big\{\big[{}^{1/2}_{1/6}\big],\, \big[{}^{1/2}_{1/2}\big] \big\},\quad
\big\{\big[{}^{1/6}_{1/2}\big],\, \big[{}^{1/2}_{1/2}\big] \big\}, \\
\big\{\big[{}^{1/3}_{1/6}\big],\, \big[{}^{\ 0}_{1/2}\big]\big\},\quad
\big\{ \big[{}^{2/3}_{1/6}\big],\, \big[{}^{\ 0}_{1/2}\big] \big\},\quad
\big\{\big[{}^{1/6}_{1/3}\big],\, \big[{}^{1/2}_{\ 0}\big]\big\},\quad
\big\{ \big[{}^{1/6}_{2/3}\big],\, \big[{}^{1/2}_{\ 0}\big] \big\},\\
\big\{\big[{}^{1/4}_{1/6}\big],\, \big[{}^{3/4}_{1/2}\big]\big\},\quad
\big\{ \big[{}^{3/4}_{1/6}\big],\, \big[{}^{1/4}_{1/2}\big] \big\},\quad
\big\{\big[{}^{1/6}_{1/4}\big],\, \big[{}^{1/2}_{3/4}\big]\big\},\quad
\big\{ \big[{}^{1/6}_{3/4}\big],\, \big[{}^{1/2}_{1/4}\big] \big\},
\end{gather*}
Directly from \eqref{MainExpr} find
\begin{align*}
&\theta'\big[{}^{1/3}_{1/6}\big] = -\frac{\pi}{3} 
\theta\big[{}^{0}_{0}\big] \theta\big[{}^{1/2}_{\ 0}\big] 
\theta^3 \big[{}^{5/6}_{1/6}\big] \theta^{-2}\big[{}^{1/3}_{1/6}\big],\\
&\theta'\big[{}^{2/3}_{1/6}\big] = -\frac{\pi}{3} 
\theta\big[{}^{0}_{0}\big] \theta\big[{}^{1/2}_{\ 0}\big] 
\theta^3 \big[{}^{1/6}_{1/6}\big] \theta^{-2}\big[{}^{2/3}_{1/6}\big], \\
&\theta'\big[{}^{1/6}_{1/3}\big] = \m \frac{\pi}{3} 
\theta\big[{}^{0}_{0}\big] \theta\big[{}^{\ 0}_{1/2}\big] 
\theta^3 \big[{}^{1/6}_{5/6}\big] \theta^{-2}\big[{}^{1/6}_{1/3}\big],\\
&\theta'\big[{}^{1/6}_{2/3}\big] = -\frac{\pi}{3} 
\theta\big[{}^{0}_{0}\big] \theta\big[{}^{\ 0}_{1/2}\big] 
\theta^3 \big[{}^{1/6}_{1/6}\big] \theta^{-2}\big[{}^{1/6}_{2/3}\big].
\end{align*}
Expressions for $\theta' \big[{}^{1/2}_{1/6}\big]$ and $\theta' \big[{}^{1/6}_{1/2}\big]$ 
are not achievable from \eqref{MainExpr}. 

Next, find  $\theta' \big[{}^{1/4}_{1/6}\big]$, that illustrates Remark~\ref{Rmk3}.
The desired theta derivative is derived from \eqref{MainExpr}
if the expression for $\theta' \big[{}^{3\varepsilon'}_{3\varepsilon}\big] = \theta' \big[{}^{3/4}_{1/2}\big]$
is known. The latter is found in Example~\ref{Ex:p4} since 
$\theta' \big[{}^{3/4}_{1/2}\big] =\imath  \theta' \big[{}^{1/4}_{1/2}\big]$. Thus,
\begin{multline*}
\theta'\big[{}^{1/4}_{1/6}\big] =
\frac{1}{3} \theta^3 \big[{}^{\ 0}_{1/6}\big] \theta^{-2}\big[{}^{1/4}_{1/6}\big]
\theta^{-1}\big[{}^{1/4}_{1/2}\big] \theta'\big[{}^{1/2}_{1/2}\big] 
-\frac{1}{3} \theta \big[{}^{1/4}_{1/6}\big]\theta^{-1}\big[{}^{1/4}_{1/2}\big]\theta'\big[{}^{1/4}_{1/2}\big] = \\
-\frac{\pi}{3} \theta\big[{}^{0}_{0}\big] \theta\big[{}^{\ 0}_{1/2}\big] \theta\big[{}^{1/2}_{\ 0}\big] 
\Big(\theta^3 \big[{}^{\ 0}_{1/6}\big] \theta^{-2}\big[{}^{1/4}_{1/6}\big]
\theta^{-1}\big[{}^{1/4}_{1/2}\big]  -  \frac{\imath}{4} \theta\big[{}^{\ 0}_{1/2}\big]^3 
\theta \big[{}^{1/4}_{1/6}\big] \theta^{-4} \big[{}^{1/4}_{1/2}\big]\Big).
\end{multline*}
Similarly find
\begin{align*}
&\theta'\big[{}^{3/4}_{1/6}\big] = 
-\imath \frac{\pi}{3} \theta\big[{}^{0}_{0}\big] \theta\big[{}^{\ 0}_{1/2}\big] \theta\big[{}^{1/2}_{\ 0}\big] 
\Big(\theta^3 \big[{}^{\ 0}_{1/6}\big] \theta^{-2}\big[{}^{3/4}_{1/6}\big]
\theta^{-1}\big[{}^{1/4}_{1/2}\big]  +  \frac{1}{4} \theta^3 \big[{}^{\ 0}_{1/2}\big] 
\theta \big[{}^{3/4}_{1/6}\big] \theta^{-4} \big[{}^{1/4}_{1/2}\big]\Big),\\
&\theta'\big[{}^{1/6}_{1/4}\big] = - \frac{\pi}{3} \theta\big[{}^{0}_{0}\big] \theta\big[{}^{\ 0}_{1/2}\big] \theta\big[{}^{1/2}_{\ 0}\big] 
\Big(\theta^3 \big[{}^{1/6}_{\ 0}\big] \theta^{-2}\big[{}^{1/6}_{1/4}\big]
\theta^{-1}\big[{}^{1/2}_{1/4}\big] - \frac{1}{4} \theta^3 \big[{}^{1/2}_{\ 0}\big]
\theta \big[{}^{1/6}_{1/4}\big] \theta^{-4} \big[{}^{1/2}_{1/4}\big] \Big),\\
&\theta'\big[{}^{1/6}_{3/4}\big] = - \frac{\pi}{3} \theta\big[{}^{0}_{0}\big] \theta\big[{}^{\ 0}_{1/2}\big] \theta\big[{}^{1/2}_{\ 0}\big] 
\Big(\theta^3 \big[{}^{1/6}_{\ 0}\big] \theta^{-2}\big[{}^{1/6}_{3/4}\big]
\theta^{-1}\big[{}^{1/2}_{1/4}\big] + \frac{1}{4} \theta^3 \big[{}^{1/2}_{\ 0}\big]
\theta \big[{}^{1/6}_{3/4}\big] \theta^{-4} \big[{}^{1/2}_{1/4}\big] \Big).
\end{align*}
\end{example}
\begin{rmk}
One can note that  the fundamental identity being applied to 
characteristics $\big[{}^{1/2}_{1/6}\big]$, $\big[{}^{1/6}_{1/2}\big]$, $\big[{}^{1/6}_{1/6}\big]$, 
and $\big[{}^{1/6}_{5/6}\big]$ becomes trivial.
Appropriate identities can be found by choosing another elliptic functions.
In \cite[sect.\,6]{M3} expressions for the theta derivatives with all mentioned characteristics are obtained,
for example Theorem 6.7 gives the following
\begin{multline*}
\theta' \big[{}^{1/2}_{1/6} \big](0,\tau) 
= \frac{\pi}{3} \theta\big[{}^{0}_{0} \big](0,\tau)\theta\big[{}^{\ 0}_{1/2} \big](0,\tau)
 \theta^4\big[{}^{1/2}_{1/6} \big](0,\tau) \theta^{-3}\big[{}^{1/2}_{1/3} \big](0,\tau) \\
 - \pi \theta^2\big[{}^{1/2}_{\ 0} \big](0,\tau) \theta\big[{}^{\ 0}_{1/6} \big](0,\tau)\theta\big[{}^{\ 0}_{1/3} \big](0,\tau)
 \theta^{-1}\big[{}^{1/2}_{1/3} \big](0,\tau).
\end{multline*}
One more identity for this theta derivative is proposed in
\cite[Theorem 3.2]{M2}:
\begin{multline*}
\theta' \big[{}^{1/2}_{1/6} \big](0,\tau) \\
= -\frac{\pi}{6} \theta\big[{}^{0}_{0} \big](0,\tau)\theta\big[{}^{\ 0}_{1/2} \big](0,\tau)
 \Big(\theta^4\big[{}^{1/2}_{1/6} \big](0,\tau) \theta^{-3}\big[{}^{1/2}_{1/3} \big](0,\tau) 
 - 3 \theta\big[{}^{1/2}_{1/3} \big](0,\tau)\Big).
\end{multline*}
\end{rmk}

\begin{example}\label{E:p65}
Combining $\mathcal{O}_{1/6}$  with $\mathcal{O}_{1/5}$, obtain the following sets of characteristics:
\begin{gather*}
\big\{\big[{}^{1/5}_{1/6}\big],\, \big[{}^{3/5}_{1/2}\big],\, \big[{}^{4/5}_{1/2}\big],\, \big[{}^{2/5}_{1/2}\big] \big\},\quad
\big\{\big[{}^{3/5}_{1/6}\big],\, \big[{}^{4/5}_{1/2}\big],\, \big[{}^{2/5}_{1/2}\big],\, \big[{}^{1/5}_{1/2}\big] \big\},\\
\big\{\big[{}^{4/5}_{1/6}\big],\, \big[{}^{2/5}_{1/2}\big],\, \big[{}^{1/5}_{1/2}\big],\, \big[{}^{3/5}_{1/2}\big] \big\},\quad
\big\{\big[{}^{2/5}_{1/6}\big],\, \big[{}^{1/5}_{1/2}\big],\, \big[{}^{3/5}_{1/2}\big],\, \big[{}^{4/5}_{1/2}\big] \big\},\\
\big\{\big[{}^{1/6}_{1/5}\big],\, \big[{}^{1/2}_{3/5}\big],\, \big[{}^{1/2}_{4/5}\big],\, \big[{}^{1/2}_{2/5}\big] \big\}, \quad
\big\{\big[{}^{1/6}_{3/5}\big],\, \big[{}^{1/2}_{4/5}\big],\, \big[{}^{1/2}_{2/5}\big],\, \big[{}^{1/2}_{1/5}\big] \big\}, \\
\big\{\big[{}^{1/6}_{4/5}\big],\, \big[{}^{1/2}_{2/5}\big],\, \big[{}^{1/2}_{1/5}\big],\, \big[{}^{1/2}_{3/5}\big] \big\}, \quad
\big\{\big[{}^{1/6}_{2/5}\big],\, \big[{}^{1/2}_{1/5}\big],\, \big[{}^{1/2}_{3/5}\big],\, \big[{}^{1/2}_{4/5}\big] \big\}.
\end{gather*}
In this case one should start with the last two characteristics, which form a period set under the action 
of multiplication by $3$, and so produce an independent system of equations. Then equation \eqref{MainExpr2}
with the first two characteristic is solved for theta derivative with the first characteristic from a set.

Considering $\big\{\big[{}^{1/5}_{1/6}\big]$, $\big[{}^{3/5}_{1/2}\big]$, $\big[{}^{4/5}_{1/2}\big]$, $\big[{}^{2/5}_{1/2}\big] \big\}$,
start with characteristics $\big\{\big[{}^{4/5}_{1/2}\big]$, $\big[{}^{2/5}_{1/2}\big] \big\}$,
which satisfy the system
\begin{align*}
&3 \theta \big[{}^{2/5}_{1/2}\big] \theta'\big[{}^{4/5}_{1/2}\big]
- \theta\big[{}^{4/5}_{1/2}\big] \theta' \big[{}^{2/5}_{1/2}\big] = 
\theta^3 \big[{}^{-11/10}_{\ -1/2}\big] \theta' \big[{}^{1/2}_{1/2}\big] \theta^{-2} \big[{}^{4/5}_{1/2}\big],\\
&3 \theta \big[{}^{1/5}_{1/2}\big] \theta'\big[{}^{2/5}_{1/2}\big]
- \theta\big[{}^{2/5}_{1/2}\big] \theta' \big[{}^{1/5}_{1/2}\big] = \rme^{2 \pi \imath/5}
\theta^3 \big[{}^{-43/10}_{\ -5/2}\big] \theta' \big[{}^{1/2}_{1/2}\big] \theta^{-2} \big[{}^{2/5}_{1/2}\big].
\end{align*}
Note that $\theta\big[{}^{4/5}_{1/2}\big] = \rme^{-2\pi \imath/5} \theta\big[{}^{1/5}_{1/2}\big]$ and 
$\theta'\big[{}^{4/5}_{1/2}\big] = \rme^{3\pi \imath/5} \theta' \big[{}^{1/5}_{1/2}\big]$.
Having applied this to the first equation, solve the system for unknown theta derivatives:
\begin{align*}
&\theta'\big[{}^{1/5}_{1/2}\big] = \frac{\pi}{10} \theta\big[{}^{0}_{0}\big] \theta\big[{}^{\ 0}_{1/2}\big] \theta\big[{}^{1/2}_{\ 0}\big]
\Big(3 \rme^{6\pi \imath/5} \theta^{3} \big[{}^{1/10}_{\; 1/2}\big] 
\theta^{-2} \big[{}^{1/5}_{1/2}\big]\theta^{-1} \big[{}^{2/5}_{1/2}\big]
+ \theta^{3} \big[{}^{3/10}_{\; 1/2}\big] \theta^{-3} \big[{}^{2/5}_{1/2}\big]\Big), \\
&\theta'\big[{}^{2/5}_{1/2}\big] = \frac{\pi}{10} \theta\big[{}^{0}_{0}\big] 
\theta\big[{}^{\ 0}_{1/2}\big] \theta\big[{}^{1/2}_{\ 0}\big]
\Big(\rme^{6\pi \imath/5} \theta^{3} \big[{}^{1/10}_{\; 1/2}\big] 
\theta^{-3} \big[{}^{1/5}_{1/2}\big]
- 3 \theta^{3} \big[{}^{3/10}_{\; 1/2}\big] \theta^{-2} \big[{}^{2/5}_{1/2}\big]\theta^{-1} \big[{}^{1/5}_{1/2}\big]\Big).
\end{align*}
At the same time, an expression for $\theta'\big[{}^{3/5}_{1/2}\big]$ is obtained, since 
$\theta'\big[{}^{3/5}_{1/2}\big] = \rme^{\pi \imath/5} \theta'\big[{}^{2/5}_{1/2}\big]$. Therefore,
$\theta'\big[{}^{1/5}_{1/6}\big]$ can be found directly from \eqref{MainExpr2} written for a pair of 
characteristics $\big[{}^{1/5}_{1/6}\big]$ and $\big[{}^{3/5}_{1/2}\big]$, that is
\begin{multline*}
\theta'\big[{}^{1/5}_{1/6}\big] = - \frac{\pi}{30} \theta\big[{}^{0}_{0}\big] \theta\big[{}^{\ 0}_{1/2}\big] \theta\big[{}^{1/2}_{\ 0}\big]
\Big(10 \theta^{3} \big[{}^{1/10}_{\; 1/6}\big] \theta^{-2} \big[{}^{1/5}_{1/6}\big]\theta^{-1} \big[{}^{2/5}_{1/2}\big] \\
+ \rme^{6\pi \imath/5} \theta^{3} \big[{}^{1/10}_{\; 1/2}\big] \theta \big[{}^{1/5}_{1/6}\big]
\theta^{-3} \big[{}^{1/5}_{1/2}\big] \theta^{-1} \big[{}^{2/5}_{1/2}\big] 
- 3 \theta^{3} \big[{}^{3/10}_{\; 1/2}\big]  \theta \big[{}^{1/5}_{1/6}\big]
\theta^{-3} \big[{}^{2/5}_{1/2}\big] \theta^{-1} \big[{}^{1/5}_{1/2}\big] \Big).
\end{multline*}
Similarly, find
\begin{align*}
&\theta'\big[{}^{1/2}_{1/5}\big] = -\frac{\pi}{10} \theta\big[{}^{0}_{0}\big] \theta\big[{}^{\ 0}_{1/2}\big] \theta\big[{}^{1/2}_{\ 0}\big]
\Big(3 \theta^{3} \big[{}^{\; 1/2}_{1/10}\big] 
\theta^{-2} \big[{}^{1/2}_{1/5}\big]\theta^{-1} \big[{}^{1/2}_{2/5}\big]
- \theta^{3} \big[{}^{\; 1/2}_{3/10}\big] \theta^{-3} \big[{}^{1/2}_{2/5}\big]\Big), \\
&\theta'\big[{}^{1/2}_{2/5}\big] = - \frac{\pi}{10} \theta\big[{}^{0}_{0}\big] 
\theta\big[{}^{\ 0}_{1/2}\big] \theta\big[{}^{1/2}_{\ 0}\big]
\Big( \theta^{3} \big[{}^{\; 1/2}_{1/10}\big] \theta^{-3} \big[{}^{1/2}_{1/5}\big]
+ 3 \theta^{3} \big[{}^{\; 1/2}_{3/10}\big] \theta^{-2} \big[{}^{1/2}_{2/5}\big]\theta^{-1} \big[{}^{1/2}_{1/5}\big]\Big),\\
&\theta'\big[{}^{1/6}_{1/5}\big] = - \frac{\pi}{30} \theta\big[{}^{0}_{0}\big] \theta\big[{}^{\ 0}_{1/2}\big] \theta\big[{}^{1/2}_{\ 0}\big]
\Big(10 \theta^{3} \big[{}^{\; 1/6}_{1/10}\big] \theta^{-2} \big[{}^{1/6}_{1/5}\big]\theta^{-1} \big[{}^{1/2}_{2/5}\big] \\
&\quad\quad - \theta^{3} \big[{}^{\; 1/2}_{1/10}\big] \theta \big[{}^{1/6}_{1/5}\big]
\theta^{-3} \big[{}^{1/2}_{1/5}\big] \theta^{-1} \big[{}^{1/2}_{2/5}\big] 
- 3 \theta^{3} \big[{}^{\; 1/2}_{3/10}\big]  \theta \big[{}^{1/6}_{1/5}\big]
\theta^{-3} \big[{}^{1/2}_{2/5}\big] \theta^{-1} \big[{}^{1/2}_{1/5}\big] \Big).
\end{align*}
In \cite[Theorem 6.2]{M2} different expressions for $\theta'\big[{}^{1/2}_{1/5}\big]$ and $\theta'\big[{}^{1/2}_{2/5}\big]$
are proposed, obtained with the help of a different elliptic function.
\end{example}

\begin{example}
Consider a more complicated example, not mentioned in the literature. Let $p=13$,
and there are four orbits on which the multiplication by $3$ acts periodically, namely:
$\mathcal{O}_{1/13} = \{1,\,3,\,9\}$, $\mathcal{O}_{2/13} = \{2,\,6,\, 5\}$, 
$\mathcal{O}_{4/13} = \{4,\,12,\,10\}$, and $\mathcal{O}_{7/13} = \{7,\,8,\,11\}$.
Now we explain how to find, for instance, $\theta' \big[{}^{\;1/13}_{12/13}\big]$. 
This characteristic produces the set
$$\mathcal{S}\big(\big[{}^{\;1/13}_{12/13}\big] \big) =
\big\{\big[{}^{\;1/13}_{12/13}\big],\, \big[{}^{\;3/13}_{10/13}\big],\, \big[{}^{9/13}_{4/13}\big]\big\}.$$
The required expression is given by \eqref{ThetaDerSol}:
\begin{multline*}
\theta'\big[{}^{\;1/13}_{12/13}\big] = \Big(26 
\theta\big[{}^{\;1/13}_{12/13}\big] \theta\big[{}^{\;3/13}_{10/13}\big] \theta\big[{}^{9/13}_{4/13}\big] \Big)^{-1}
\theta'\big[{}^{1/2}_{1/2}\big] \theta\big[{}^{\;1/13}_{12/13}\big] \times \\ \times \Big(
9 \rme^{-6\pi \imath/13} \theta^3 \big[{}^{\ \ 9/26}_{-35/26}\big] \theta\big[{}^{9/13}_{4/13}\big] \theta^{-2}\big[{}^{\;1/13}_{12/13}\big]
+ 3\rme^{-6\pi \imath/13} \theta^3 \big[{}^{\ \ \ 1/26}_{-131/26}\big] \theta\big[{}^{\; 1/13}_{12/13}\big] 
\theta^{-2}\big[{}^{\;3/13}_{10/13}\big] \\
+ \rme^{-10\pi \imath/13} \theta^3 \big[{}^{\ -23/26}_{-419/26}\big] \theta\big[{}^{\; 3/13}_{10/13}\big] 
\theta^{-2} \big[{}^{9/13}_{4/13}\big]\Big) \\
= -\frac{\pi}{26} \theta\big[{}^{0}_{0}\big] \theta\big[{}^{\ 0}_{1/2}\big] \theta\big[{}^{1/2}_{\ 0}\big] \Big(
9 \rme^{-8 \pi \imath/13} \theta^3 \big[{}^{\ 9/26}_{17/26}\big] \theta^{-1}\big[{}^{\;3/13}_{10/13}\big]  
\theta^{-2}\big[{}^{\;1/13}_{12/13}\big] \\
+ 3\rme^{2 \pi \imath/13} \theta^3 \big[{}^{\ 1/26}_{25/26}\big] \theta\big[{}^{\; 1/13}_{12/13}\big] 
\theta^{-1} \big[{}^{9/13}_{4/13}\big] \theta^{-3}\big[{}^{\;3/13}_{10/13}\big] 
+ \rme^{-7 \pi \imath/13} \theta^3 \big[{}^{\ 3/26}_{23/26}\big] 
\theta^{-3} \big[{}^{9/13}_{4/13}\big]\Big).
\end{multline*}
Similarly,
\begin{align*}
&\theta'\big[{}^{\;3/13}_{10/13}\big] = 
-\frac{\pi}{26} \theta\big[{}^{0}_{0}\big] \theta\big[{}^{\ 0}_{1/2}\big] \theta\big[{}^{1/2}_{\ 0}\big] \Big(
9 \rme^{2 \pi \imath/13} \theta^3 \big[{}^{\ 1/26}_{25/26}\big] \theta^{-1}\big[{}^{9/13}_{4/13}\big]  
\theta^{-2}\big[{}^{\;3/13}_{10/13}\big] \\
&\quad + 3\rme^{-7 \pi \imath/13} \theta^3 \big[{}^{\ 3/26}_{23/26}\big] \theta\big[{}^{\;3/13}_{10/13}\big]
\theta^{-1} \big[{}^{\;1/13}_{12/13}\big] \theta^{-3} \big[{}^{9/13}_{4/13}\big]
+ \rme^{-8 \pi \imath/13} \theta^3 \big[{}^{\ 9/26}_{17/26}\big] 
\theta^{-3} \big[{}^{\; 1/13}_{12/13}\big]  \Big), \\
&\theta'\big[{}^{9/13}_{4/13}\big] = 
-\frac{\pi}{26} \theta\big[{}^{0}_{0}\big] \theta\big[{}^{\ 0}_{1/2}\big] \theta\big[{}^{1/2}_{\ 0}\big] 
\Big(
9 \rme^{-7 \pi \imath/13} \theta^3 \big[{}^{\ 3/26}_{23/26}\big] 
 \theta^{-1}\big[{}^{\;1/13}_{12/13}\big]  
\theta^{-2}\big[{}^{9/13}_{4/13}\big] \\
&\quad + 3 \rme^{-8 \pi \imath/13} \theta^3 \big[{}^{\ 9/26}_{17/26}\big]  \theta\big[{}^{9/13}_{4/13}\big]
\theta^{-1} \big[{}^{\;3/13}_{10/13}\big] \theta^{-3} \big[{}^{\;1/13}_{12/13}\big]
+ \rme^{2 \pi \imath/13} \theta^3 \big[{}^{\ 1/26}_{25/26}\big]
\theta^{-3} \big[{}^{\; 3/13}_{10/13}\big]  \Big).
\end{align*}
\end{example}

\begin{rmk}
Looking at examples one can see that all terms in expressions for theta derivatives
are homogeneous in theta constants of degree $3$.
\end{rmk}

\end{document}